\documentclass{amsart}%
\usepackage{amsfonts}
\usepackage{amsmath}
\usepackage{amssymb}
\usepackage{graphicx}%
\providecommand{\U}[1]{\protect\rule{.1in}{.1in}}
\newtheorem{theorem}{Theorem}
\theoremstyle{plain}

\newtheorem{corollary}{Corollary}

\newtheorem{definition}{Definition}
\newtheorem{example}{Example}

\newtheorem{lemma}{Lemma}

\newtheorem{proposition}{Proposition}
\newtheorem{remark}{Remark}

\numberwithin{equation}{section}
\begin{document}
\title[$\delta$-$n$-ideals of commutative rings]{$\delta$-$n$-ideals of commutative rings}
\author{Ece Yetkin Celikel}
\address{Department of Electrical Electronics Engineering, Faculty of Engineering,
Hasan Kalyoncu University, Gaziantep, Turkey.}
\email{ece.celikel@hku.edu.tr, yetkinece@gmail.com.}
\author{Gulsen Ulucak}
\address{Department of Mathematics, Faculty of Science, Gebze Technical University,
Gebze, Kocaeli, Turkey.}
\email{gulsenulucak@gtu.edu.tr.}
\thanks{This paper is in final form and no version of it will be submitted for
publication elsewhere.}
\subjclass[2000]{ Primary 13A15.}
\keywords{$\delta$-$n$-ideal, quasi $n$-ideal, $n$-ideal, $\delta$-primary ideal}

\begin{abstract}
Let $R$ be a commutative ring with nonzero identity, and $\delta
:\mathcal{I(R)}\rightarrow\mathcal{I(R)}$ be an ideal expansion where
$\mathcal{I(R)}$ the set of all ideals of $R$. In this paper, we introduce the
concept of $\delta$-$n$-ideals which is an extension of $n$-ideals in
commutative rings. We call a proper ideal $I$ of $R$ a $\delta$-$n$-ideal if
whenever $a,b\in R$ with$\ ab\in I$ and $a\notin\sqrt{0}$, then $b\in
\delta(I)$. For example, $\delta_{1}$ is defined by $\delta_{1}(I)=\sqrt{I}.$
A number of results and characterizations related to $\delta$-$n$-ideals are
given. Furthermore, we present some results related to quasi $n$-ideals which
is for the particular case $\delta=\delta_{1}.$

\end{abstract}
\maketitle


\section{Introduction}

Throughout this paper, we assume that all rings are commutative with non-zero
identity$.$ Since prime ideals have an important place in commutative algebra,
various generalizations of prime ideals have studied by many authors. D. Zhao
\cite{Zhao} introduced the concept of expansions of ideals and $\delta
$-primary ideals of commutative rings. Let $R$ be a ring. By $\mathcal{I(R)}$,
we denote the set of all ideals of $R.$ According to his paper, a function
$\delta:\mathcal{I(R)}\rightarrow\mathcal{I(R)}$ is an is an ideal expansion
if it assigns to each ideal $I$ of $R$ to another ideal $\delta(I)$ of the
same ring with the following properties: $I\subseteq\delta(I)$ and if
$I\subseteq J$ for some ideals $I,J$ of $R$, then $\delta(I)\subseteq$
$\delta(J).$ For example, $\delta_{0}$ is the identity function where
$\delta_{0}(I)=I$ for all ideal $I$ of $R,$ and $\delta_{1}$ is defined by
$\delta_{1}(I)=\sqrt{I}.$ For the other examples, consider the functions
$\delta_{+}$ and $\delta_{\ast}$ of $\mathcal{I(R)}$ defined with $\delta
_{+}(I)=I+J$ where $J\in\mathcal{I(R)}$ and $\delta_{\ast}(I)=(I:P)$ where
$P\in\mathcal{I(R)}$ for all $I\in\mathcal{I(R)}$, respectively. Recall from
\cite{Zhao} that an ideal expansion $\delta$ is said to be intersection
preserving if it satisfies $\delta(I\cap J)=\delta(I)\cap\delta(J)$ for any
ideals $I,J$ of $R$. He called a $\delta$-primary ideal $I$ of $R$ if $ab\in
I$ and $a\notin I$ for some $a,b\in R$ imply $b\in\delta(I)$. As a recent
study, \cite{Tekir}, authors defined the concept of $n$-ideals. A proper ideal
$I$ of $R$ is called $n$-ideal if whenever $a,b\in R$ and $ab\in I$, then
$a\in\sqrt{0}$ or $b\in I$.

The aim of this article is to introduce $\delta$-$n$-ideals which is an
extention of $n$-ideals of commutative rings and to give relations with some
classical ideals such as prime, $\delta$-primary, $n$-ideal. We call a proper
ideal $I$ of $R$ a $\delta$-$n$-ideal if whenever $a,b\in R$ with$\ ab\in I$
and $a\notin\sqrt{0}$, then $b\in\delta(I).$ In particular, if $\delta
=\delta_{1}$, then it is said to be a quasi $n$-ideal of $R.$ It is clear that
every $n$-ideal is a $\delta$-$n$-ideal for all ideal expansions $\delta.$ We
start with Example \ref{e3} is given to show that $\delta$-$n$-ideal and
$n$-ideal are different concepts. Also, a prime ideal needs not to be a
$\delta$-$n$-ideal (see Example \ref{e1}). Among many results in this paper,
in Proposition \ref{cx}, we obtain some certain conditions for a prime ideal
is to be a $\delta$-$n$-ideal. In Theorem \ref{t1}, we conclude equivalent
characterizations for $\delta$-$n$-ideals. In Theorem \ref{qu}, we discuss
rings of which every proper ideal is a $\delta$-$n$-ideal. We show in
Proposition \ref{max} that a maximal quasi $n$-ideal of $R$, is a prime ideal
of $R.$ In Proposition \ref{t4}, we show that an integral domain has no
nonzero $\delta$-$n$-ideal for expansion of ideals $\delta$ of $R$ with
$\delta(I)\neq R$ for all $I\in\mathcal{I(R)}.$ Also, it is shown in Theorem
\ref{von} that if $\delta(0)=0,$ then $R$ is a field if and only if $R$ is a
von Neumann regular ring and $\{0\}$ is a $\delta$-$n$-ideal. Furthermore, we
investigate $\delta$-$n$-ideals under various contexts of constructions such
as homomorphic images, direct products, localizations and in idealization
rings. (See Proposition \ref{tt}, \ref{loc}, Remark \ref{di}, and Proposition
\ref{i}).

For the sake of completeness, we give some definitions which we will need
throughout this study. For a proper ideal $I$ a ring $R$, $\sqrt{I}$ denotes
the radical of $I$ defined by $\{r\in R$ : there exists $n\in%
\mathbb{N}
$ with $r^{n}\in I\}$ and for $x\in R,$ by $(I:x),$ we denote the set of
$\{r\in R$ : $rx\in I\}$. Let $M$ be a unitary $R$-module. Recall that the
idealization $R(+)M=\{(r,m):r\in R,$ $m\in M$\} is a commutative ring with the
addition $(r_{1},m_{1})+(r_{2},m_{2})=(r_{1}+r_{2},m_{1}+m_{2})$ and
multiplication $(r_{1},m_{1})(r_{2},m_{2})=(r_{1}r_{2},r_{1}m_{2}+r_{2}m_{1})$
for all $r_{1},r_{2}\in R;\ m_{1},m_{2}\in M$. For an ideal $I$ of $R$ and a
submodule $N$ of $M$, it is well-known that $I(+)N$ is an ideal of $R(+)M$ if
and only if $IM\subseteq N$ \cite{AnWi} and \cite{Huc}. We recall also from
\cite{AnWi} that $\sqrt{I(+)N}=\sqrt{I}(+)M$. For the other notations and
terminologies that are used in this article, the reader is referred to
\cite{Huc}.

\section{Properties of $\delta$-$n$-ideals}

\begin{definition}
Given an expansion $\delta$ of ideals, a proper ideal $I$ of a ring $R$ is
called a $\delta$-$n$-ideal if whenever $a,b\in R$ and$\ ab\in I$ and
$a\notin\sqrt{0}$, then $b\in\delta(I).$
\end{definition}

It is clear that a proper ideal $I$ of $R$ is a $\delta_{0}$-$n$-ideal if and
only if $I$ is an $n$-ideal, and an $n$-ideal is a $\delta$-$n$-ideal.
However, the following example shows that the converse of this implication is
not true in general.

\begin{example}
\label{e3}Let $I=(x^{3})R_{1}$ be an ideal of $R_{1}=\mathbb{Z}_{4}[X]$. Let
$R=R_{1}/I$. Define the expansion function of $\mathcal{I(R)}$ with
$\delta(K)=K+\frac{(2,x)R_{1}}{I}$ and let $J=(x+1)R_{1}/I$. We show that $J$
is a $\delta$-$n$-ideal but not a $n$-ideal of $R$. Since $((x+1)+I)(1+I)\in
J$ but $((x+1)+I)\notin\sqrt{0_{R}}=\frac{(2,x)R_{1}}{I}$ and $(1+I)\notin J,$
$J$ is not an $n$-ideal of $R$. Note that $\delta(J)=\frac{(x+1)R_{1}}%
{I}+\frac{(2,x)R_{1}}{I}$. Thus $1+I\in\delta(J)$, that is, $\delta(J)=R$.
Thus $J$ is a $\delta$-$n$-ideal.
\end{example}

\begin{proposition}
\label{td} Let $\delta$ be an expansion of ideals of $R$ and $I$ a proper
ideal of $R$ with $\delta(I)\neq R.$ If $I$ is a $\delta$-$n$-ideal of $R$,
then $I\subseteq\sqrt{0}.$
\end{proposition}

\begin{proof}
Assume that $I\nsubseteq\sqrt{0}.$ Then there is an element $a\in R$ with
$a\in I\backslash\sqrt{0}.$ Since$\ a=a\cdot1\in I$ and $a\notin\sqrt{0}$, we
conclude $1\in\delta(I)$, a contradiction. Thus $I\subseteq\sqrt{0}.$
\end{proof}

Note that if the converse of Proposition \ref{td} is not satisfied in general.
For example, consider the ideal $I=\{0\}$ of $R=\mathbb{Z}_{6}.$ Put
$\delta=\delta_{0}$ or $\delta=\delta_{1}.$ Since $2\cdot3\in I$ but neither
$2\in\sqrt{0}$ nor $3\in\delta(I)$, $I$ is not a $\delta$-$n$-ideal of $R.$

In the following result, we clarify the relationships between $\delta$-primary
ideals and $\delta$-$n$-ideals.

\begin{proposition}
\label{p1} Let $I\subseteq\sqrt{0}$ be a proper ideal of a ring $R$ and
$\delta$ be an expansion of ideals of $R$. If $I$ is a $\delta$-primary ideal
of $R$, then $I$ is a $\delta$-$n$-ideal of $R$. The converse is also true if
$I=\sqrt{0}$.
\end{proposition}

\begin{proof}
Suppose that $a,b\in R$ with $ab\in I$ and $a\notin\sqrt{0}$. Since $I$ is a
$\delta$-primary and clearly $a\notin I$, we have $a\in\delta(I),$ as needed.
In particular, it is clear that $\sqrt{0}$ is a $\delta$-primary ideal if and
only if $\sqrt{0}$ is a $\delta$-$n$-ideal.
\end{proof}

We show in the next example that a prime ideal needs not to be a $\delta$%
-$n$-ideal of $R$ in general.

\begin{example}
\label{e1} Let $\delta_{+}:\mathcal{I(R)}\rightarrow\mathcal{I(R)}$ be an
expansion of ideals of $R=\mathbb{Z}$ defined by $\delta_{+}(J)=J+q\mathbb{Z}$
where $q$ is prime integer with $(p,q)=1$. Consider the ideal $I=p\mathbb{Z}$
where $p$ is prime integer of the ring $R=\mathbb{Z}$. Then $I$ is a
$\delta_{+}$-$n$-ideal of $R$ that is neither $n$-ideal, $\delta_{0}$%
-$n$-ideal nor $\delta_{1}$-$n$-ideal of $R$. Indeed, $p\cdot1\in I$ but
$p\notin\sqrt{0}$ and $1\notin\delta_{0}(I)=I$ and also $1\notin\delta
_{1}(I)=\sqrt{I}$.
\end{example}

We justify the conditions for a prime ideal and $\delta$-primary is to be a
$\delta$-$n$-ideal of $R$ in the next result.

\begin{proposition}
\label{cx}Let $\delta$ be an expansion of ideals of $R.$ Then the following
are hold:
\end{proposition}

\begin{enumerate}
\item Let $I$ be a $\delta$-primary ideal of $R$ with $\delta(I)\neq R.$ Then
$I$ is a $\delta$-$n$-ideal of $R$ if and only if $I\subseteq\sqrt{0}$.

\item Let $I$ be a prime ideal of $R$ with $\delta(I)\neq R$. Then $I$ is a
$\delta$-$n$-ideal of $R$ if and only if $I=\sqrt{0}.$
\end{enumerate}

\begin{proof}
(1) From Proposition \ref{td} and \ref{p1}, the result is clear.

(2) Since $I$ is prime, $\sqrt{0}\subseteq I.$ Then the equality holds by
Proposition \ref{td}. Conversely, let $I=\sqrt{0}.$ Then $I$ is an $n$-ideal
of $R$ by \cite[Proposition 2.8]{Tekir}, and so, $I$ is $\delta$-$n$-ideal.
\end{proof}

The next theorem gives a characterization for $\delta$-$n$-ideal of $R$ in
terms of the ideals of $R$.

\begin{theorem}
\label{t1} For a proper ideal $I$ of $R$ and an expansion of fuction $\delta$,
the following statements are equivalent:
\end{theorem}

\begin{enumerate}
\item $I$ is a $\delta$-$n$-ideal of $R.$

\item $(I:a)\subseteq\sqrt{0}$ for all $a\in R-\delta(I).$

\item If $aJ\subseteq I$ for some $a\in R$ and an ideal $J$ of $R$, then
$a\in\sqrt{0}$ or $J\subseteq\delta(I)$.

\item If $JK\subseteq I$ for some ideals $J$ and $K$ of $R$ implies
$J\cap(R-\sqrt{0})=\emptyset$ or $K\subseteq\delta(I).$
\end{enumerate}

\begin{proof}
(1)\ $\Rightarrow$(2) Let $b\in(I:a)$. Since $I$ is $\delta$-$n$-ideal and
$a\notin\delta(I)$, we have $b\in\sqrt{0}.$ Thus $(I:a)\subseteq\sqrt{0}.$

(2)$\Rightarrow$(3) Assume that $aJ\subseteq I$ but $J\not \subseteq
\delta(I).$ Then there exists an element $j$ of $J$ with $j\not \in \delta
(I)$. Hence $a\in(I:j)$ which implies that $a\in\sqrt{0}$ by (2).

(3)$\Rightarrow$(4) Suppose that $JK\subseteq I$ and $J\cap(R-\sqrt{0}%
)\neq\emptyset$. Then there is $a\in R$ with $a\in J\cap(R-\sqrt{0})$. By (1),
we have $K\subseteq\delta(I)$ since $aK\subseteq I$ and $a\notin\sqrt{0}$.
\newline(4)\ $\Rightarrow$(1) Let $ab\in I$ for some $a,b\in R.$Put $J=(a)$
and $K=(b).$ So we have the result by our assumption. $a\in\sqrt{0}.$ Thus $I$
is a $\delta$-$n$-ideal of $R$.
\end{proof}

Next, we justify some equivalent conditions for rings of which every proper
ideal is $\delta$-$n$-ideal.

\begin{theorem}
\label{qu}For every expansion function $\delta$ of ideals of $R$, the
following statements are equivalent:
\end{theorem}

\begin{enumerate}
\item Every proper principal ideal is a $\delta$-$n$-ideal of $R$.

\item Every proper ideal is a $\delta$-$n$-ideal of $R$.

\item $\sqrt{0}$ is the unique prime ideal of $R$.

\item $R$ is a quasi local ring with maximal element $M=\sqrt{0}.$
\end{enumerate}

\begin{proof}
(1)$\Rightarrow$(2) Let $I$ be a proper ideal of $R$ and $a,b\in R$ with
$ab\in I$ and $a\notin\sqrt{0}$. Put $J=(ab)$. Since $J$ is a $\delta$-$n$
-ideal, we conclude that $b\in\delta(J)\subseteq\delta(I)$, as needed.

(2)$\Rightarrow$(3) Suppose that $I$ is a prime ideal of $R$. Then it is
$\delta$-$n$-ideal by our assumption. Thus $I=\sqrt{0}$ by Proposition
\ref{cx}.

(3)$\Rightarrow$(4) It is clear.

(4)$\Rightarrow$(1) Suppose that $(R,\sqrt{0})$ is a quasi local ring. Then
every element of $R$ is either unit or nilpotent. Let $I=(x)$ be a principal
ideal and let $a,b\in R$, $ab\in(x)$ and $a\notin\sqrt{0}.$ Then $a$ is unit
and so $b\in(x)=I\subseteq\delta(I)$. Thus $I$ is a $\delta$-$n$-ideal.
\end{proof}

\begin{proposition}
\label{t4}Let $R$ be an integral domain and $\delta$ be an expansion of
$\mathcal{I(R)}$ such that $\delta(I)\neq R$ for every $I\in\mathcal{I(R)}$.
Then $\{0\}$ is the only $\delta$-$n$-ideal of $R$.
\end{proposition}

\begin{proof}
Suppose that $R$ is an integral domain. Then $\sqrt{0}=0$ and $0$ is clearly a
$\delta$-$n$-ideal of $R$. Now, assume that $I$ is nonzero $\delta$-$n$-ideal
of $R$. Then $I\subseteq\sqrt{0}=0$ by Proposition \ref{td} which is a contradiction.
\end{proof}

Recall from \cite{Jac} that a von Neumann regular ring is a ring such that for
all $a\in R$, there exists an $x\in R$ satisfying $a=a^{2}x.$ In particular,
$R$ is a Boolean ring if for all $a\in R$, $a=a^{2}.$

\begin{theorem}
\label{von}Let $\delta$ be an ideal expansion of ideals of $R$ with
$\delta(0)=\{0\}.$ Then $R$ is a field if and only if $R$ is a von Neumann
regular ring and $\{0\}$ is a $\delta$-$n$-ideal.
\end{theorem}

\begin{proof}
Suppose that $R$ is a von Neumann regular ring and $\{0\}$ is a $\delta$%
-$n$-ideal. Then clearly $\sqrt{0}=\{0\}$. We show that every nonzero element
$a$ of $R$ is unit. Since $R$ is von Neumann regular, there exists $x\in R$
such that $a=a^{2}x$. Hence $a(1-ax)=0$. Since $a\notin\sqrt{0}$, we conclude
that $1-ax\in$ $\delta(0)=0$. Thus $ax=1$, as needed. Therefore $R$ is a
field. The converse part is clear by \cite[Theorem 2.15]{Tekir}.
\end{proof}

Since a Boolean ring is a von Neumann regular ring, Theorem \ref{von} is also
valid for Boolean rings.

\begin{lemma}
\label{:}Let $\delta$ be an expansion of $\mathcal{I(R)}$. If $I$ is a
$\delta$-$n$-ideal of $R$ such that $(\delta(I):x)\subseteq\delta(I:x)\neq R$
for all $x\in R\backslash\delta(I)$, then $(I:x)$ is a $\delta$-$n$-ideal of
$R$. In particular, if $I$ is a quasi $n$-ideal of $R$, then $(I:x)$ is a
quasi $n$-ideal of $R$ for all $x\in R\backslash\delta(I)$.
\end{lemma}

\begin{proof}
Suppose that $ab\in(I:x)$ and $a\notin\sqrt{0}.$ Since $abx\in I$ and $I$ is
$\delta$-$n$-ideal, we conclude that $bx\in\delta(I).$ Thus $b\in
(\delta(I):x)\subseteq\delta(I:x)$, so we are done. For the "in particular
case", we just need to show that the inclusion $(\delta_{1}(I):x)\subseteq
\delta_{1}(I:x)$ is satisfied for all $x\in R\backslash\delta_{1}(I)$. Let
$a\in(\delta_{1}(I):x)$. Then $ax\in\delta_{1}(I).$ Since clearly $a^{n}%
x^{n}\in I$ for some positive integer $n$, $I$ is a $\delta_{1}$-$n$-ideal and
$x^{n}\notin\delta_{1}(I)$, we conclude $a^{n}\in\delta_{1}(I),$ that is,
$a\in\delta_{1}(I)\subseteq\delta_{1}(I:x)$. Thus we have the inclusion and
the result comes from the general case above.
\end{proof}

\begin{proposition}
\label{max}Let $\delta$ be an expansion of $\mathcal{I(R)}$. If $I$ is a
maximal $\delta$-$n$-ideal of $R$ with $(\delta(I):x)\subseteq\delta(I:x)\neq
R$ where $x\in R\backslash\delta(I)$, then $I=\sqrt{0}$ is a prime ideal of
$R.$ In particular, if $I$ is a maximal quasi $n$-ideal of $R$, then
$I=\sqrt{0}$ is a prime ideal of $R.$
\end{proposition}

\begin{proof}
Suppose that $I$ is a maximal $\delta$-$n$-ideal of $R.$ We show that $I$ is
prime. Let $ab\in I$ and $a\notin I.$ Hence $(I:a)$ is a $\delta$-$n$-ideal of
$R$ by Lemma \ref{:}. Thus $(I:a)=I$ since $I$ is a maximal $\delta$%
-$n$-ideal. It means $b\in I$, and thus $I$ is a prime ideal of $R$. From
Proposition \ref{cx} (2), we conclude that $I=\sqrt{0}.$ The "in particular"
case is clear from the proof of Lemma \ref{:}.
\end{proof}

So, we are ready for the following result.

\begin{theorem}
\label{eq}Let $\delta$ be an expansion of $\mathcal{I(R)}$ with $(\delta
(J):x)\subseteq\delta(J:x)\neq R$ for all ideal $J$ of $R$ and $x\in
R\backslash\delta(J)$. Then the following statements are equivalent:
\end{theorem}

\begin{enumerate}
\item There exists an $\delta$-$n$-ideal of $R.$

\item $\sqrt{0}$ is a prime ideal of $R.$

\item $\sqrt{0}$ is a $\delta$-primary ideal of $R.$
\end{enumerate}

\begin{proof}
(1) $\Rightarrow$ (2) Let $I$ is a $\delta$-$n$-ideal of $R$ and $W=\{J:J$ is
an $n$-ideal of $R\}$. Then $W$ is a nonempty partially ordered set by the set
inclusion. Take a chain $I_{1}\subseteq I_{2}\subseteq\cdots\subseteq
I_{n}\subseteq.\cdots$ of $W$. We show that $I=\cup_{i\in\Lambda}I_{i}$ is a
$\delta$-$n$-ideal of $R$. Suppose that $ab\in I$ and $a\notin I$ for some
$a,b\in R$. Then $ab\in I_{k}$ for some $k\in\Lambda$. Since $a\notin I_{k}$
and $I_{k}$ is $\delta$-$n$-ideal, we conclude that $b\in\sqrt{0}.$ Thus
$I=\cup_{i\in\Lambda}I_{i}$ is an upper bound of the chain. So, there exists a
maximal element $M$ of $W$ by the Zorn's Lemma. It follows $M=\sqrt{0}$ from
Proposition \ref{max}. Converse part is clear from \cite[Corollary
2.9(i)]{Tekir}.

(2) $\Rightarrow$ (3) is clear.

(3)$\Rightarrow$ (3) It follows from Proposition \ref{p1}.
\end{proof}

\begin{proposition}
Let $\delta$ be an expansion function of $\mathcal{I(R)}$ and $I$ be proper
ideal of $R$ with $\delta(\delta(I))=\delta(I)$ (in particular, let
$\delta=\delta_{1}$)$.$ Then the following hold:
\end{proposition}

\begin{enumerate}
\item If $I$ is $\delta$-$n$-ideal and $a\notin\sqrt{0}$, then $\delta
(I:a)=\delta(I).$

\item $\delta(I)$ is $n$-ideal if and only if $\delta(I)$ is $\delta$-$n$-ideal.

\item If $IK=JK$ and $I,J$ are $\delta$-$n$-ideals of $R$ with $\delta
(\delta(J))=\delta(J)$ and $K\cap(R-\sqrt{0})\neq\emptyset$ for some ideal $K$
of $R$, then $\delta(I)=\delta(J)$.

\item If $IK$ and $I$ are $\delta$-$n$-ideals of $R$ with $\delta
(\delta(IK))=\delta(IK)$ and $K\cap(R-\sqrt{0})\neq\emptyset$ for some ideal
$K$ of $R$, then $\delta(IK)=\delta(I)$.
\end{enumerate}

\begin{proof}
(1)\ Let $I$ be $\delta$-$n$-ideal and $a\notin\sqrt{0}.$ Note that
$I\subseteq(I:a)$ and so $\delta(I)\subseteq\delta(I:a).$ Let $x\in(I:a).$
Then $x\in\delta(I)$ since $xa\in I$ and $a\notin\sqrt{0}.$ Thus
$(I:a)\subseteq\delta(I)$. We get $\delta(I:a)\subseteq\delta(\delta
(I))=\delta(I).$ Hence we conclude the equality.

(2)\ It is clear from our assumption.

(3) Note that $IK=JK\subseteq I,J$. Then we have $J\subseteq\delta(I)$ since
$JK\subseteq I$ and $K\cap(R-\sqrt{0})\neq\emptyset$ and also $I\subseteq
\delta(J)$ in a similar way. Thus $\delta(I)=\delta(J)$ as $\delta
(\delta(I))=\delta(I)$ and $\delta(\delta(J))=\delta(J).$

(4) It is clear that $\delta(IK)\subseteq\delta(I)$ since $IK\subseteq I.$ We
have $I\subseteq\delta(IK)$ since $IK\subseteq IK$ and $K\cap(R-\sqrt{0}%
)\neq\emptyset.$ Thus $\delta(IK)=\delta(I)$ by our assumption.
\end{proof}

An element $a\in R$ is said to be $\delta$-nilpotent if $a\in\delta(0).$

\begin{proposition}
\label{zero}Let $\delta$ be an expansion function of $\mathcal{I(R)}$. Then
$\sqrt{0}$ is a $\delta$-$n$-ideal of $R$ if and only if every zero-divisor of
the quotient ring $R/\sqrt{0}$ is $\delta_{q}$-nilpotent$.$
\end{proposition}

\begin{proof}
Suppose that $\overline{a}=a+\sqrt{0}$ is a zero-divisor of $R/\sqrt{0}.$ Then
$\overline{a}\overline{b}=(a+\sqrt{0})(b+\sqrt{0})=\sqrt{0}$ for some
$\sqrt{0}\neq\overline{b}\in R/\sqrt{0}.$ It means $ab\in\sqrt{0}$ but
$b\notin\sqrt{0}.$ Since $\sqrt{0}$ is a $\delta$-$n$-ideal, we conclude
$a\in\delta(\sqrt{0})$. Hence $\overline{a}=a+\sqrt{0}\in\delta(\sqrt
{0})/\sqrt{0}$. Now consider the natural epimorphism $\Pi:R\rightarrow
R/\sqrt{0}.$ Note that $\Pi$ is a $\delta\delta_{q}$-epimorphism. We have
$\delta(\sqrt{0})/\sqrt{0}=\delta(\Pi^{-1}(0_{R/\sqrt{0}}))=\Pi^{-1}%
(\delta_{q}(0_{R/\sqrt{0}})).$ Since $\Pi$ is epimorphism, then $\delta
(\sqrt{0})/\sqrt{0}=\Pi(\delta(\sqrt{0}))=\delta(0_{R/\sqrt{0}}).$ Thus
$\overline{a}\in\delta_{q}(0_{R/\sqrt{0}})$; so $\overline{a}$ is $\delta_{q}%
$-nilpotent$.$ Conversely, Suppose that $ab\in\sqrt{0}$ and $a\notin\sqrt{0}$
for some $a,b\in R$. Then $\overline{a}\overline{b}=\sqrt{0}=0_{R/\sqrt{0}}$
but $\overline{a}\neq0_{R/\sqrt{0}}.$ It means that $\overline{b}$ is a zero
divisor of $R/\sqrt{0}.$ Then $\overline{b}$ is a $\delta_{q}$-nilpotent from
our assumption. Hence $\overline{b}\in\delta_{q}(0_{R/\sqrt{0}})=\delta
(\sqrt{0})/\sqrt{0}.$ So $b+\sqrt{0}=c+\sqrt{0}$ for some $c\in\delta(\sqrt
{0}).$ It follows $b-c\in\sqrt{0}\subseteq\delta(\sqrt{0}).$ Thus
$b=(b-c)+c\in\delta(\sqrt{0});$ so $\sqrt{0}$ is a $\delta$-$n$-ideal of $R$.
\end{proof}

\begin{proposition}
\label{d}Let $\delta$ and $\gamma$ be expansion functions of $R$ and $I$ be a
proper ideal of $R.$ Then
\end{proposition}

\begin{enumerate}
\item If $\delta(I)$ is an $n$-ideal of $R$, then $I$ is a $\delta$-$n$-ideal
of $R$. The converse of this inclusion is also true if $\delta=\delta_{1}.$

\item Let $\delta(I)\subseteq\gamma(I)$ for all ideals $I$ of $R.$ If $I$ is a
$\delta$-$n$-ideal of $R$, then $I$ is a $\gamma$-$n$-ideal of $R$.

\item If $\gamma(I)$ is a $\delta$-$n$-ideal of $R,$ then $I$ is a
$\delta\circ\gamma$-$n$-ideal of $R.$
\end{enumerate}

\begin{proof}
(1) Suppose that $ab\in I$ and $a\notin\sqrt{0}$ for some $a,b\in R$. Since
$I\subseteq\delta(I)$ and $\delta(I)$ is an $n$-ideal, we conclude $b\in
\delta(I).$ Thus $I$ is a $\delta$-$n$-ideal of $R$. Conversely, suppose that
$\delta=\delta_{1}$. Let $ab\in\delta_{1}(I)$ and $a\notin\sqrt{0}$. Then
$a^{n}b^{n}\in I$ for some $n\geq1$ and clearly $a^{n}\notin\sqrt{0}$. Since
$I$ is a $\delta_{1}$-$n$-ideal, we have $a^{n}\in\delta_{1}(I)$. Thus
$a\in\delta_{1}(I)$, as required.

(2) It is obvious.

(3) Assume that $\gamma(I)$ is a $\delta$-$n$-ideal of $R.$ Let $ab\in I$ for
some $a,b\in R$ and $a\notin\sqrt{0}.$ Then since $I\subseteq\gamma(I)$, we
have $ab\in\gamma(I).$ Since $\gamma(I)$ is a $\delta$-$n$-ideal of $R,$
$b\in\delta(\gamma(I))=\delta\circ\gamma(I)$, we are done.
\end{proof}

In Example \ref{e1}, we show that $I=p\mathbb{Z}$ is a $\delta_{+}$-$n$-ideal
of $\mathbb{Z}$ where $p$ is prime integer of the ring $R=\mathbb{Z}$. But
$\delta_{+}(I)$ is not an $n$-ideal of $\mathbb{Z}$ since it is not a proper.
Hence it can be seen that the converse of Proposition \ref{d} (1) may not be true.

\begin{proposition}
\label{t5}Let $\delta$ be an ideal expansion of $\mathcal{I(R)}$ and $I$ be a
proper ideal of $R$ and$\sqrt{\delta(I)}=\delta(\sqrt{I})$. If $I$ is a
$\delta$-$n$-ideal of $R$, then $\sqrt{I}$ is a $\delta$-$n$-ideal of $R$. In
particular, $I$ is a quasi $n$-ideal of $R$ if and only if $\sqrt{I}$ is a
$n$-ideal of $R$.
\end{proposition}

\begin{proof}
Let $a,b\in R$ with $ab\in I$ and $a\notin\sqrt{0}.$ Then $(ab)^{n}=a^{n}%
b^{n}\in I$ for some positive integer $n$. Since $I$ is $\delta$-$n$-ideal and
$a^{n}\notin\sqrt{0},$ we have $b^{n}\in\delta(I).$ Hence $b\in\sqrt
{\delta(I)}=\delta(\sqrt{I})$. Thus $\sqrt{I}$ is a $\delta$-$n$-ideal of $R$.
The "in particular" case follows from Proposition \ref{d}.
\end{proof}

\begin{proposition}
Let $I,J$ and $K$ proper ideals of $R$ with $J\subseteq K\subseteq I$. If $I$
is a $\delta$-$n$-ideal of $R$ and $\delta(J)=\delta(I)$, then $K$ is a
$\delta$-$n$-ideal of $R$.
\end{proposition}

\begin{proof}
Assume that $I$ is a $\delta$-$n$-ideal of $R$ and $\delta(J)=\delta(I)$. Let
$ab\in K$ for some $a,b\in R.$ Then $a\in\sqrt{0}$ or $b\in\delta(K)$ since
$K\subseteq I$ and $\delta(J)=\delta(I)=\delta(K).$ Thus, $K$ is a $\delta
$-$n$-ideal of $R$.
\end{proof}

An ideal expansion $\delta$ is intersection preserving if it satisfies
$\delta(I\cap J)=\delta(I)\cap\delta(J)$ for any $I,J\in\mathcal{I(R)}$
\cite{Zhao}.

\begin{proposition}
\label{int}Let $\delta$ be an ideal expansion which preserves intersection.
Then the following statements are hold:
\end{proposition}

\begin{enumerate}
\item If $I_{1},I_{2},...,I_{n}$ are $\delta$-$n$-ideals of $R$, then $I=%
{\displaystyle\bigcap\limits_{i=1}^{n}}
I_{i}$ is a $\delta$-$n$-ideal of $R$.

\item Let $I_{1},I_{2},...,I_{n}$ be of $R$ such that $\delta(I_{i})$'s are
non-comparable prime ideals of $R$. If $%
{\displaystyle\bigcap\limits_{i=1}^{n}}
I_{i}$ is a $\delta$-$n$-ideal of $R$, then $I_{i}$ is a $\delta$-$n$-ideal of
$R$ for all $i=1,2,...,n.$
\end{enumerate}

\begin{proof}
(1) Let $ab\in I$ and $b\notin\delta(I)$ for some $a,b\in R$. Since
$\delta(I)=\cap_{i=1}^{n}\delta(I_{i}),$ $b\notin\delta(I_{k})$ for some
$k\in\{1,...,n\}.$ It follows $a\in\sqrt{0}.$ Thus $I$ is a $\delta$-$n$-ideal
of $R$.

(2) Suppose that $ab\in I_{k}$ and $a\notin\sqrt{0}$ for some $k\in
\{1,2,...,n\}.$ Choose an element $x\in\left(
{\displaystyle\prod\limits_{\substack{i=1\\i\neq k}}^{n}}
I_{i}\right)  \backslash\delta(I_{k})$. Hence, $abx\in%
{\displaystyle\bigcap\limits_{i=1}^{n}}
I_{i}$. Since $%
{\displaystyle\bigcap\limits_{i=1}^{n}}
I_{i}$ is a $\delta$-$n$-ideal, we have $bx\in\delta\left(
{\displaystyle\bigcap\limits_{i=1}^{n}}
I_{i}\right)  =%
{\displaystyle\bigcap\limits_{i=1}^{n}}
\delta(I_{i})\subseteq\delta(I_{k})$ which implies $b\in\delta(I_{k})$ as
$\delta(I_{k})$ is prime, so we are done.
\end{proof}

Let $R$ and $S$ be two commutative rings and $\delta,\gamma$ be expansion
functions of $\mathcal{I(R)}$ and $\mathcal{I(S)},$ respectively. Then a ring
homomorphism $f:R\rightarrow S$ is called a $\delta\gamma$-homomorphism if
$\delta(f^{-1}(J))=f^{-1}(\gamma(J))\ $for all ideal $J\ $of $S.\ $Let
$\gamma_{1}\ $be a radical operation on ideals of $S\ $and $\delta_{1}$ be a
radical operation on ideals of $R.\ $ A homomorphism from $R\ $to $S$ is an
example of $\delta_{1}\gamma_{1}$-homomorphism. Additionaly, if $f$ is a
$\delta\gamma$-epimorphism and $I\ $is an ideal of $R\ $containing
$\ker(f),\ $then $\gamma(f(I))=f(\delta(I)).$

\begin{proposition}
\label{tt} Let $f:R\rightarrow S$ be a $\delta\gamma$-homomorphism, where
$\delta$ and $\gamma$ are expansion functions of $\mathcal{I(R)}$ and
$\mathcal{I(S)}$, respectively. Then the following hold:
\end{proposition}

\begin{enumerate}
\item Let $\ f$ be a monomorphism. If $J\ $is a $\gamma$-$n$-ideal of $S$,
then $f^{-1}\left(  J\right)  $ is a $\delta$-$n$-ideal of $R.\ $

\item Suppose that$\ f$ is an epimorphism and $I$ is a proper ideal of
$R$\ with $\ker(f)\subseteq I.$ If $I\ $is a $\delta$-$n$-ideal of $R$, then
$f\left(  I\right)  $ is a $\gamma$-$n$-ideal of $S.$
\end{enumerate}

\begin{proof}
\begin{enumerate}
\item Let $ab\in f^{-1}(J)$ for $a,b\in R$. Then $f(ab)=f(a)f(b)\in J$, which
implies $f(a)\in\sqrt{0_{S}}$ or $f(b)\in\gamma(J)$. If $f(a)\in\sqrt{0_{S}},$
then $a\in\sqrt{0_{R}}$ as $\ker(f)=\{0\}$. If $f(b)\in\gamma(J)$, then we
have $b\in f^{-1}(\gamma(J))=\delta(f^{-1}(J))$ since $f$ is $\delta\gamma
$-homomorphism. Thus $f^{-1}(J)$ is a $\delta$-$n$-ideal of $R.$

\item Suppose that $a,b\in S$ with $ab\in f(I)$ and $a\notin\sqrt{0_{S}}.$
Since $f$ is an epimorphism, there exist $x,y\in R$ such that $a=f(x)$ and
$b=f(y).$ Then clearly we have $x\notin\sqrt{0_{R}}$ as $a\notin\sqrt{0_{S}}.$
Since $\ker(f)\subseteq I,$ $ab=f(xy)\in f(I)$ implies that $(0\neq xy\in I)$
$xy\in I$. Thus $y\in\delta(I)$; and so $b=f(y)\in f(\delta(I)).$ On the other
hand, since $\gamma(f(I))=f(\delta(I)),$ we have $b\in\gamma(f(I)).$ Thus
$f(I)$ is a $\gamma$-$n$-ideal of $S$.
\end{enumerate}
\end{proof}

Let $\delta$ be an expansion function of $\mathcal{I(R)}$ and $I\ $ be an
ideal of $R$. \ Then the function $\delta_{q}:R/I\rightarrow R/I$ is defined
by $\delta_{q}(J/I)=\delta(J)/I\ $\ for all ideals $I\subseteq J,\ $becomes an
expansion function of $\mathcal{I(R/I)}.$

\begin{corollary}
\label{/}Let $\delta$ be an expansion function of $\mathcal{I(R)}$ and
$J\subseteq I$ proper ideals of $R.$ Then the followings hold:

\begin{enumerate}
\item If $I$ is a $\delta$-$n$-ideal of $R,$ then $I/J$ is a $\delta_{q}$%
-$n$-ideal of $R/J.$

\item $I/J$ is a $\delta_{q}$-$n$-ideal of $R/J$ and $J\subseteq\sqrt{0_{R}},$
then $I$ is a $\delta$-$n$-ideal of $R.$

\item $I/J$ is a $\delta_{q}$-$n$-ideal of $R/J$ and $J$ is a $\delta$%
-$n$-ideal of $R$ where $\delta(J)\neq R$, then $I$ is a $\delta$-$n$-ideal of
$R.$

\item Let $K$ be a subring of $R$ with $S\nsubseteq I$. Then $S\cap I$ is a
$\delta$-$n$-ideal of $R.$
\end{enumerate}
\end{corollary}

\begin{proof}
\begin{enumerate}
\item Consider the natural homomorphism $\pi:R\rightarrow R/J$. By Proposition
\ref{tt} (2), we have $I/J$ is a $\delta_{q}$-$n$-ideal of $R/J$ since
$\ker(\pi)\subseteq I.$

\item Let $I/J $ be a $\delta_{q}$-$n$-ideal of $R/J $ and $J\subseteq
\sqrt{0_{R}}.$ Assume that $ab\in I $ and $a\notin\sqrt{0} $ for some $a,b\in
R. $ Then $ab+J=(a+J)(b+J)\in I/J $ and $a+J\notin\sqrt{0_{R/J}}. $ By our
assumption, $b+J\in\delta_{q}(I/J) =\delta(I)/J$, that is, $b\in\delta(I). $

\item It is clear by (2) and Proposition \ref{td}.

\item Let the injection $i:S\rightarrow R$ be defined with $i(a)=a$ for every
$a\in S.$ Then the proof is clear by Proposition \ref{tt}(1).\newline
\end{enumerate}
\end{proof}

Let $I$ be a proper ideal of\ a ring $R.$ Recall that $I$ is said to be
superfluous if there is no proper ideal $J$ of $R$ such that $I+J=R$. In the
following, by $J(R),$ we denote the Jacobson radical of $R.$

\begin{lemma}
\label{sup}Any $\delta$-$n$-ideal a ring $R$ with $\delta(I)\neq R$ is superfluous.
\end{lemma}

\begin{proof}
Let $I$ be a $\delta$-$n$-ideal of $R$ with $\delta(I)\neq R.$ Assume that
there exists a proper ideal $J$ of $R$ with $I+J=R$. Then $1=a+b$ for some
$a\in I$ and $b\in J$ and so $1-b\in I\subseteq\sqrt{0}\subseteq J(R)$ by
Proposition \ref{td}. Thus $b\in J$ is a unit and so, we get $J=R,$ a contradiction.
\end{proof}

\begin{proposition}
\label{sum}Let $I$ and $J$ be $\delta$-$n$-ideals of a ring $R$ such that
$\delta(I)\neq R$ and $\delta(J)\neq R$. Then $I+J$ is a $\delta$-$n$-ideal of
$R$.
\end{proposition}

\begin{proof}
Let $I$ and $J$ be $\delta$-$n$-ideals of a ring $R$ such that $\delta(I)\neq
R$ and $\delta(J)\neq R.$ Since they are superfluous by Lemma \ref{sup},
$I+J\neq R$. Hence, $I\cap J$ is a $\delta$-$n$-ideal by Proposition
\ref{int}. Also, $I/(I\cap J)$ is a $\delta_{q}$-$n$-ideal of $R/(I\cap J)$ by
Corollary \ref{/} (1). Now, by the isomorphism $I/(I\cap J)\cong(I+J)/J$,
$(I+J)/J$ is a $\delta_{q}$-$n$-ideal of $R/J$. Therefore, Corollary \ref{/}
(3) implies that $I+J$ is a $\delta$-$n$-ideal of $R$.
\end{proof}

Let $S$ be a multiplicatively closed subset of $R$. Note that $\delta_{S}$ is
an expansion function of $\mathcal{I(S}^{-1}\mathcal{R)\ }$such that
$\delta_{S}(S^{-1}I)=S^{-1}(\delta(I))$ where $\delta$ is an expansion
function of $R.$ By $Z_{I}(R),$ we\ denote the set of $\{r\in R|rs\in I$ for
some $s\in R\backslash I\}$ where $I$ is a proper ideal of $R.$

\begin{proposition}
\label{loc}Let $S$ be a multiplicatively closed subset of $R$ and $\delta$ an
expansion function of $R.$
\end{proposition}

\begin{enumerate}
\item If $I$ is a $\delta$-$n$-ideal of $R$ with $I\cap S=\emptyset,$ then
$S^{-1}I$ is a $\delta_{S}$-$n$-ideal of $S^{-1}R.$

\item Let $S\cap Z(R)=S\cap Z_{\delta(I)}(R)=\emptyset$. If $S^{-1}I$ is a
$\delta_{S}$-$n$-ideal of $S^{-1}R,$ then $I$ is a $\delta$-$n$-ideal of $R.$
\end{enumerate}

\begin{proof}
(1) Suppose that $\frac{a}{s}\frac{b}{t}\in S^{-1}I$ and $\frac{a}{s}%
\notin\sqrt{0_{S^{-1}R}}$ for some $a,b\in R$ and $s,t\in S.$ Then there is
$u\in S$ with $abu\in I.$ Thus $bu\in\delta(I)$ since $a\notin\sqrt{0}.$ Hence
$\frac{b}{t}=\frac{bu}{tu}\in S^{-1}(\delta(I))=\delta_{S}(S^{-1}I)$.
Consequently, $S^{-1}I$ is a $\delta_{S}$-$n$-ideal of $S^{-1}R.$

(2) Let $a,b\in R$ with $ab\in I$. Then $\frac{a}{1}\frac{b}{1}\in S^{-1}I$
implies that either $\frac{a}{1}\in\sqrt{0_{S^{-1}R}}$ or $\frac{b}{1}%
\in\delta_{S}(S^{-1}I)$. If $\frac{a}{1}\in\sqrt{0_{S^{-1}R}}$, then
$ua^{n}=0$ for some $u\in S$ and a positive integer $n.$ Since $S\cap
Z(R)=\emptyset$, we conclude $a^{n}=0$ and $a\in\sqrt{0}.$ If $\frac{b}{1}%
\in\delta_{S}(S^{-1}I)=S^{-1}(\delta(I))$, then $vb\in\delta(I)$ for some
$v\in S.$ Our assumption $S\cap Z_{\delta(I)}(R)=\emptyset$ implies that
$b\in\delta(I)$, as needed.
\end{proof}

An element $a\in R$ is called regular if $ann(a)=0.$ Let $r(R)$ be the set of
all regular elements of $R.$ Note that $r(R)$ is a multiplicatively closed
subset of $R$. From \cite[Proposition 2.20]{Tekir}, we obtain that if $I$ is a
$\delta_{r(R)}$-$n$-ideal of $R_{r(R)}$, then $I^{c}$ is $\delta$-$n$-ideal of
$R$.\newline

\begin{remark}
\label{di}Let $R=R_{1}\times R_{2}$ be a commutative ring where $R_{i}$ is a
commutative ring with nonzero identity for each $i\in\{1,2\}.$ Every ideal $I$
of $R$ is the form of $I=I_{1}\times I_{2}$ where $I_{i}$ is an ideal of
$R_{i}$ for all $i\in\{1,2\}.$ Let $\delta_{i}\ $be an expansion function of
$\mathcal{I(R}_{i}\mathcal{)\ }$for each $i\in\{1,2\}.$ Let $\delta_{\times}$
be a function of $\mathcal{I(R)}$, which is defined by $\delta_{\times}%
(I_{1}\times I_{2})=\delta_{1}(I_{1})\times\delta_{2}(I_{2})\ $. Then
$\delta_{\times}$ is an expansion function of $\mathcal{I(R)}$. If $\delta
_{i}(I_{i})\neq R_{i}$ for some $i\in\{1,2\}$, then $R$ has no a
$\delta_{\times}$-$n$-ideal. Suppose that $I=I_{1}\times I_{2}$ is a
$\delta_{\times}$-$n$-ideal of $R$ where $I_{i}$ is an ideal of $R_{i}$ for
$i\in\{1,2\}.$ As $(1,0)(0,1)\in I$ and $(1,0),(0,1)\notin\sqrt{0_{R}},$ then
we have $(1,0),(0,1)\in\delta_{\times}(I)$. Thus $\delta_{\times}%
(I)=\delta_{1}(I_{1})\times\delta_{2}(I_{2})=R_{1}\times R_{2},$ a
contradiction.\newline
\end{remark}

Let $R(+)M$ be the idealization where $M$ is an $R$-module. For an expansion
function $\delta$ of $R,$ define $\delta_{(+)}$ as $\delta_{(+)}%
(I(+)N)=\delta(I)(+)M$ for some ideal $I(+)N\ $of $R(+)M$. It is clear that
$\delta_{(+)}$ is an expansion function of $R(+)M.$ Next, we characterize
$\delta$-$n$-ideals in any idealization ring $R(+)M$.

\begin{proposition}
\label{i}Let $I$ be an ideal of of a ring $R$ and $N$ be a submodule of an
$R$-module $M$. Then $I$ is a $\delta$-$n$-ideal of $R$ if and only if $I(+)N$
is a $\delta_{(+)}$-$n$-ideal of $R(+)M.$
\end{proposition}

\begin{proof}
Let $I$ be a $\delta$-$n$-ideal of $R$. Assume that $(r,m)(s,m^{\prime})\in
I(+)N$ and $(s,m^{\prime})\notin\sqrt{0}(+)M$ for some $(r,m)(s,m^{\prime})\in
R(+)M.$ Then $s\in\delta(I)$ since $rs\in I$ and $s\notin\sqrt{0}$. Thus
$(s,m^{\prime})\in\delta(I)(+)M=\delta_{(+)}(I(+)M).$ Conversely, \ suppose
that $I(+)N$ is a $\delta_{(+)}$-$n$-ideal of $R(+)M.$ Let $r,s\in R$ with
$rs\in I$ and $s\notin\sqrt{0}$. Hence, we get $(r,m)(s,m^{\prime})\in I(+)N$
and clearly $(s,m^{\prime})\notin\sqrt{0}(+)M$ which follows $(r,m)\in
\delta_{(+)}(I(+)M)$, and $r\in\delta(I)$, as required.
\end{proof}


\begin{thebibliography}{9}                                                                                                %


\bibitem {AB}D. D. Anderson and M. Batanieh, Generalizations of prime ideals,
Comm. Algebra, \textbf{36} (2008), 686--696.

\bibitem {AnWi}D. D. Anderson and M. Winders, (2009). Idealization of a
module. Journal of Commutative Algebra, 1(1), 3-56.

\bibitem {Gilmer}R. Gilmer, Multiplicative ideal theory, Queen Papers Pure
Appl. Math. 90, Queen's University, Kingston, 1992.

\bibitem {Huc}J. Huckaba, Rings with zero-divisors, Marcel Dekker, NewYork/ Basil,1988.

\bibitem {Tekir}U. Tekir, S. Koc, K.H. Oral, $n$-ideals of commutative rings,
Filomat 31 (10) (2017) 2933-2941.

\bibitem {Zhao}D. Zhao, $\delta$-primary ideals of commutative rings,
Kyungpook Math. J.,41 (2001),17--22.

\bibitem {Jac}Jacobson, N. Basic Algebra II, 2nd ed. New York: W. H. Freeman, 1989.
\end{thebibliography}
\end{document}